\documentclass[a4paper,12pt]{extarticle}
\usepackage{preamble}

\title{Generic countably infinite groups}
\author{M\'arton Elekes, Bogl\'arka Geh\'er, Krist\'of Kanalas, Tam\'as K\'atay, Tam\'as Keleti}

\begin{document}

\maketitle

\begin{abstract}
Countably infinite groups (with a fixed underlying set) constitute a Polish space $\calg$ with a suitable metric, hence the Baire category theorem holds in $\calg$. We study isomorphism invariant subsets of $\calg$, which we call group properties. We say that the \textbf{generic} countably infinite group is of property $\calp$ if $\calp$ is comeager in $\calg$.

We prove that every group property with the Baire property is either meager or comeager. We show that there is a comeager elementary equivalence class in $\calg$ but every isomorphism class is meager. We prove that the generic group is algebraically closed, simple, not finitely generated and not locally finite. We show that in the subspace of Abelian groups the generic group is isomorphic to the unique countable, divisible torsion group that contains every finite Abelian group.

We sketch the model-theoretic setting in which many of our results can be generalized. We briefly discuss a connection with infinite games.
\end{abstract}

\begin{spacing}{1}
\tableofcontents
\end{spacing}

\newpage

\section{Introduction}

\textbf{Motivation.} Generic properties in the sense of Baire category have been intensively studied in almost all branches of mathematics for a century. For example, there is a vast literature on the behaviour of the generic continuous function.

More recently, E.~Akin, M.~Hurley and J.~A.~Kennedy \cite{AKIN} investigated the dynamics of generic homeomorphisms of compact spaces. A.~Kechris and C.~Rosendal \cite{KECHRIS2} studied the automorphism groups of homogeneous countable structures. Among numerous other results they characterized when an automorphism group admits a comeager conjugacy class. Even more recently, M.~Doucha and M.~Malicki \cite{DOUCHA} examined generic representations of discrete countable groups in Polish groups.

From the model-theoretic point of view P.~J.~Cameron \cite{CAMERON} studied the genericity of relational structures. M.~Pouzet and B.~Roux \cite{POUZET} extended the study to a general setting and obtained results on metric spaces and transition systems. Z.~Kabluchko and K.~Tent \cite{TENT1}, \cite{TENT2} studied the genericity of Fra\"issé limits. A.~Kruckman \cite{KRUCKMAN} investigated infinitary limits of classes of finite structures. Among several other results he presented theorems on the generic limits \cite[Chapter 2]{KRUCKMAN}.

%An important difference between \cite{KRUCKMAN} and our approach is that $\Dir_K$ contains only locally finite structures. Since algebraically closed groups contain free subgroups by \ref{t.alg_c_word_p}, it follows that the set of locally finite groups is meager in $\calg$.

%They worked with the Baire space of all countably infinite, not finitely generated structures whose age is contained in a given Fra\"issé class $\calc$ and which are defined on a fixed countably infinite universe. In this space the isomorphism class of the Fra\"issé limit of $\calc$ is comeager.

Genericity is often studied via infinite games. For example, W.~Kubiś \cite{KUBIS2} introduced a variant of the Banach-Mazur game played in partially ordered sets. In a recent paper \cite{KUBIS1} A.~Krawczyk and W.~Kubiś studied another variant played with finitely generated structures.

\textbf{Main goal.} The main goal of this paper is to initiate the study of generic properties of countably infinite groups in the sense of Baire category.

\textbf{Setup.} We can fix a common underlying set, say $\nat$, for all countably infinite groups. We define a natural topology on the set of multipilication tables of countably infinite groups (see Section~\ref{s.the_space}). With this topology it is a Polish space $\calg$. In particular, the Baire category theorem holds in $\calg$. We reserve the term \emph{group property} for isomorphism-invariant subsets of $\calg$. Now it makes sense to study generic group properties in $\calg$.

\textbf{The main results and the organization of the paper.} Section~\ref{s.prelim} provides essential preliminaires in algebra, logic and descriptive set theory. In Section~\ref{s.the_space} we introduce the space in which our work takes place. In Section~\ref{s.dense} we give a simple sufficient condition for a group property to be dense in $\calg$. As an easy direct application we show that simpleness is generic.

In Section~\ref{s.0-1} we present a 0-1 law for group properties. Namely, we prove that \emph{every group property that has the Baire property} (BP) as a subset of $\calg$ \emph{is either meager or comeager}. As the BP is a very weak condition, less formally this theorem says that every sensible group property is either meager or comeager. We show that isomorphism classes and group properties defined by first-order formulas in the language of group theory have the BP (in fact, they are Borel). Therefore every isomorphism class is either meager or comeager, and \emph{there is a comeager elementary equivalence class} in $\calg$ (with respect to the language of group theory).

One may suspect that the reason behind results of Section~\ref{s.0-1} is that there is a comeager isomorphism class in $\calg$. To settle this problem we need algebraically closed groups, which are the group theoretic analogues of algebraically closed fields. In Section~\ref{s.alg_closed} we show that \emph{algebraic closedness is a generic property}.

In Section~\ref{s.embed_isom} we prove that a group $G$ is generically embeddable (that is, it can be embedded into a comeager set of groups) if and only if $G$ can be embedded into every algebraically closed group. By the results of B.~H.~Neumann \cite{NEUMANN3}, H~Simmons \cite{SIMMONS} and A.~Macintyre \cite{MACINTYRE2} this is also equivalent to the following property: every finitely generated subgroup of $G$ has solvable word problem. This result reveals a connection between our topology and the important and well-studied word problem. Moreover, the characterization of embeddability allows us to prove that \emph{every isomorphism class is meager} in $\calg$.

In Section~\ref{s.abelian_groups} we study the subspace of Abelian groups and we show that there is a comeager isomorphism class in this subspace. Thus one may say, informally, that \emph{there is a generic countably infinite Abelian group}. In Section~\ref{s.games} we connect our results with the theory of infinite games. In Section~\ref{s.problems} we list four problems representing four directions in which the study may be continued.

\textbf{Generalizations.} Although many of our results have straightforward generalizations in model theory, we formulate them as theorems about groups to help readers not familiar with model theory. The general setting is presented at the end of Section~\ref{s.the_space}. Throughout the paper we make remarks about model-theoretic generalizations. It will take another paper to accurately generalize the proofs and study the general problem of isomorphism classes.

\section{Preliminaries}\label{s.prelim}

We beleive this paper could be interesting for researchers of several fields of mathematics. Therefore, we include essential, sometimes basic preliminaries in algebra, logic and descriptive set theory. The reader may skip the familiar parts.

\subsection{Algebra}\label{ss.algebra}

We will need the following notion from combinatorial group theory:

\begin{defi}
A finitely generated group $G$ has \textbf{solvable word problem} if there exists a Turing machine that decides for every word in the generators of $G$ whether it represents the identity element. It is easy to see that the solvability of the word problem is independent of the choice of the finite generating set.
\end{defi}

Let $F$ be a free group on the infinite generating set $X=\{x_1,x_2,\dots\}$ and $G$ be any group. Recall (see \cite[Chapter~3]{MASSEY}) that an element of the free product $F\star G$ is a word whose letters are from $X$ and $G$. Intuitively, in a word $w\in F\star G$ letters from $X$ are variables and letters from $G$ are parameters.

Let $E$ and $I$ be finite subsets of $F\star G$. We view $E$ as a set of \textbf{equations} and $I$ as a set of \textbf{inequations}. A \textbf{solution} of the system $(E,I)$ in $G$ is a homomorphism $f: F\to G$ such that the unique homomorphism $\wtilde f: F\star G\to G$ extending both $f$ and the identity of $G$ maps every $e\in E$ to $1_G$ and does not map any $i\in I$ to $1_G$.

The system $(E,I)$ is \textbf{consistent with $G$} if it has a solution in a bigger group $H\geq G$. That is, if there exists a group $H$
and an embedding $h:G\to H$ such that for the unique 
homomorphism $\wtilde h:F\star G\to F\star H$ extending both $h$ and the identity of $F$, the system $\left(\wtilde h (E),\wtilde h(I)\right)$ has a solution in $H$.

\begin{defi}\label{d.alg_closed}
The group $G$ is \textbf{algebraically closed} if every finite system $(E,I)$ of equations and inequations that is consistent with $G$ has a solution in $G$. 
\end{defi}

\begin{remark}
Some authors prefer the term existentially closed that comes from model theory and reserve the term algebraically closed for the case when one does not allow inequations in the definition. However, B.~H.~Neumann proved in \cite{NEUMANN1} that the two notions coincide except for the trivial group. Since we study only infinite groups, the terminology will cause no confusion.
\end{remark}

By a standard closure argument one easily verifies the following.

\begin{theorem}[Scott, \cite{SCOTT}]\label{t.emb_alg_closed}
Every countable group can be embedded into a countable algebraically closed group. In particular, algebraically closed groups exist.
\end{theorem}

\begin{defi}\label{d.homogeneous}
A countably infinite group $G$ is \textbf{homogeneous} if any isomorphism between two finitely generated subgroups of $G$ extends to an automorphism of $G$. If this extension can always be chosen to be an inner automorphism, then $G$ is \textbf{strongly homogeneous}.
\end{defi}

Homogeneous groups have the following property that can be proved by a straightforward back-and-forth argument, see \cite[Lemma 7.1.4]{HODGES}.

\begin{prop}
If $G$ and $H$ are homogeneous countably infinite groups and they have the same finitely generated subgroups (up to isomorphism), then $G\cong H$.
\end{prop}

The following proposition is also known, see \cite[Lemma 1]{MACINTYRE1}. However, we present the fairly short proof because we will refer to it.

\begin{prop}\label{p.alg_closed_strongly_hom}
Countably infinite algebraically closed groups are strongly homogeneous.
\end{prop}

The proof is based on the existence of HNN extensions:

\begin{theorem}[Higman--Neumann--Neumann, \cite{HNN}]\label{t.HNN}
If $G$ is a group and $\alpha: H\to K$ is an isomorphism between two subgroups of $G$, then there is a group $G_\alpha\geq G$ with only one new generator $t$ such that the conjugation by $t$ is an automorphism of $G_\alpha$ extending $\alpha$.
\end{theorem}

\begin{proof}[Proof of Proposition~\ref{p.alg_closed_strongly_hom}]
Let $G$ be a countably infinite algebraically closed group. Let $H$ and $K$ be finitely generated subgroups of $G$ and $\alpha:H\to K$ be an isomorphism. Let $\{h_1,\dots,h_n\}$ be a generating set for $H$. Consider the following system of equations:
\begin{equation}\label{e.conjugation}
    x^{-1}h_1x=\alpha(h_1),\dots,x^{-1}h_nx=\alpha(h_n).
\end{equation}
By the HNN extension construction (\ref{e.conjugation}) is consistent with $G$. (More accurately, the system $E=(x^{-1}h_1x(\alpha(h_1))^{-1},\ldots,x^{-1}h_nx(\alpha(h_n))^{-1})$ is consistent with $G$.) Then it has a solution $g$ in $G$ because $G$ is algebraically closed. The conjugation by $g$ coincides with $\alpha$ on the generating set $\{h_1,\dots,h_n\}$ of $H$; therefore the conjugation by $g$ is an inner automorphism of $G$ extending $\alpha$.
\end{proof}

Finally we cite a very nice result that connects the word problem and algebraically closed groups. It is due to B.~H.~Neumann, H.~Simmons and A.~Macintyre, see \cite{NEUMANN3}, \cite{SIMMONS} and \cite{MACINTYRE2}.

\begin{theorem}\label{t.alg_c_word_p}
A finitely generated group has solvable word problem if and only if it is embeddable into every algebraically closed group.
\end{theorem}

\begin{remark}\label{r.prop_of_ac}
It is well-known that algebraically closed groups are simple \cite{NEUMANN1} and not finitely generated \cite{NEUMANN3}. Since they contain free subgroups by Theorem~\ref{t.alg_c_word_p} they are not locally finite either.
\end{remark}

\subsection{Logic}\label{ss.logic}

In this section we present essential logical preliminaries that can be found in any logic or model theory textbook. See, for example, \cite{HODGES}. Needless to say, we do not have space for a complete and precise development of the basic notions.

The \textbf{alphabet} of a first-order language $L$ is a set containing two types of symbols:
\begin{itemize}
    \item \textbf{Logical symbols}: Negation ($\lnot$), conjunction ($\land$), existential quantifier ($\exists$); an infinite set of variables ($x, y, z,\ldots$); equality symbol ($=$); parentheses, brackets, commas for punctuation. We use further logical symbols such as $\forall$, $\lor$ and $\implies$ for convenience.
    \item \textbf{Non-logical symbols}: function symbols $(f_i)_{i\in I}$ and relation symbols $(r_j)_{j\in J}$ for some index sets $I$ and $J$.
\end{itemize}
To define an alphabet one must determine the \textbf{arity} (that is, the number of arguments) of each function symbol and relation symbol. The $0$-ary function symbols are the \textbf{constant symbols}.

An $L$\textbf{-structure} $M$ is a set $\dom (M)$ (called the \textbf{universe} of $M$) together with an $n_i$-ary function $f_i^M:\dom(M)^{n_i}\to \dom(M)$ for each $i\in I$ and an $m_j$-ary relation $r_j^M\subseteq \dom (M)^{m_j}$ for each $j\in J$, where $n_i$ is the arity of the symbol $f_i$ and $m_j$ is the arity of the symbol $r_j$. These functions and relations are the \textbf{interpretations} of the corresponding symbols in $M$. Let $M,N$ be $L$-structures. Then $M$ is a \textbf{substructure} of $N$ (and $N$ is an extension of $M$) if $\dom(M)\subseteq\dom(N)$ and the restriction of every function $f_i^N$ and relation $r_j^N$ to $\dom(M)$ coincides with the corresponding function $f_i^M$ and relation $r_j^M$ of $M$.

\textbf{Terms.} Every variable and constant symbol is a term. If $t_1,\dots,t_n$ are terms and $f_i$ is an $n$-ary function symbol, then $f_i(t_1,\dots,t_n)$ is a term.

\textbf{Formulas.} If $t_1$ and $t_2$ are terms, then $t_1=t_2$ is a formula. If $t_1,\dots,t_n$ are terms and $r$ is an $n$-ary relation symbol, then $r(t_1,\dots,t_n)$ is a formula. These two types of formulas are the \textbf{atomic formulas}. If $\varphi_1, \varphi_2$ are formulas, then $\lnot \varphi_1$, $\varphi_1\land \varphi_2$ and $\exists x\varphi_1$ are formulas. Again, we allow other logical connectives and the $\forall$ quantifier for convenience. Formulas without quantifiers are called \textbf{quantifier-free}. Thus $\exists x \varphi$ is a formula even if $x$ does not occur in $\varphi$. In practice one does not encounter such formulas.

\textbf{Bound and free variables.} In a quantifier-free formula every variable is \textbf{free}. The free variables of a formula of the form $\exists x\varphi$ (or $\forall x\varphi$) are the free variable of $\varphi$ except $x$. The non-free variables of a formula are called \textbf{bound} variables. A formula without free variables is a \textbf{sentence}. Usually $\varphi(x_1,\dots,x_n)$ denotes a formula with free variables $x_1,\dots,x_n$.

\textbf{Evaluation and truth.} For an $L$-structure $M$ a \textbf{variable assignment} is a function that maps every variable to an element of $M$. \emph{Given a variable assignment $e$} we define the \textbf{evaluation} of terms and formulas recursively:

The variable $x$ evaluates to $e(x)$. The term $f(t_1,\dots,t_n)$ evaluates to $f^M(t_1[e],\dots,t_n[e])$, where $t_i[e]$ is the evaluation of $t_i$.

A formula $\varphi$ of the form $t_1=t_2$ is \textbf{true} in $M$ if $t_1[e]=t_2[e]$. For a relation symbol $r$ in $L$ a formula $\varphi$ of the form $r(t_1,\dots,t_n)$ is true in $M$ if $(t_1[e],\dots,t_n[e])\in r^M$. The formula $\varphi_1\land\varphi_2$ is true in $M$ if both $\varphi_1$ and $\varphi_2$ are true in $M$. The formula $\lnot\varphi$ is true in $M$ if $\varphi$ is not true in $M$. The formula $\exists x\varphi$ is true in $M$ if there is an element $a\in M$ such that $\varphi$ is true in $M$ with the modified variable assignment $e'$ that maps $x$ to $a$ and equals $e$ elsewhere.

We write $M\models\varphi[e]$ if the formula $\varphi$ is true in $M$ with the variable assignment $e$. We also say that $M$ \textbf{satisfies} $\varphi[e]$. The evaluated variables occuring in a given formula are the \textbf{parameters}. It is easy to prove that the truth value of a sentence does not depend on the variable assignment (we write $M\models\varphi$ if the sentence $\varphi$ is true in $M$). A set of sentences is a \textbf{theory}. An $L$-structure $M$ is a \textbf{model} of the $L$-theory $\Gamma$ if every element of $\Gamma$ is true in $M$. A theory is \textbf{consistent} if it has a model.

A formula is \textbf{existential} if it is of the form $\exists x_1,\ldots\exists x_n\varphi(x_1,\dots,x_n,y_1,\dots,y_k)$, where $\varphi$ is quantifier-free. The \textbf{existential closure} of a formula $\varphi(x_1,\dots,x_n)$ is $\exists x_1,\dots,x_n\varphi(x_1,\dots,x_n)$. It is easy to see that for every existential formula $\varphi$ and evaluation $e$ if $M\models\varphi[e]$, then $\varphi[e]$ is true in every extension of $M$ as well.

The natural generalization of algebraic closedness is existential closedness that is defined relative to a fixed theory $\Gamma$. A model $M$ of $\Gamma$ is \textbf{existentially closed} if for every extension $M'$ of $M$ the following holds: if an existential formula $\varphi$ with parameters from $M$ is true in $M'$, then it is true in $M$ as well. See \cite[Section 8.1]{HODGES} for an introduction to existentially closed models.

Two $L$-structures are \textbf{elementarily equivalent} if they satisfy the same $L$-sentences. Two $L$-structures are \textbf{isomorphic} if there is a bijection between them that preserves relations and functions.
It is easy to see that isomorphic $L$-structures are elementarily equivalent.

A formula $\varphi$ is in \textbf{prenex normal form} (PNF) if it is written as a string of quantifiers and bound variables followed by a quantifier-free part. Every first-order formula is equivalent to some formula in PNF. We define the formula classes $\forall_n$ and $\exists_n$ for formulas in PNF. Let both $\forall_0$ and $\exists_0$ denote the class of quantifier-free formulas in PNF. A formula in PNF is in $\forall_{n+1}$ if it is of the form $\forall x_1\dots \forall x_k\varphi$ with $\varphi\in\exists_n$. Similarly, a formula in PNF is in $\exists_{n+1}$ if it is of the form $\exists x_1\dots\exists x_k\varphi$ with $\varphi\in\forall_n$. We allow $k=0$, hence we have $\forall_n,\exists_n\subseteq\forall_{n+1},\exists_{n+1}$ for every $n\in\nat$.

\textbf{Example.} The alphabet of the language $L$ of group theory contains the logical symbols and a binary function symbol $m$ and a constant symbol $1$. (There are variants: we could eliminiate $1$ which is used only for convenience; we could introduce a $1$-ary function symbol for inverses.) Thus an $L$-structure is just a set with a binary operation and an element marked with the symbol $1$. Terms are expressions like $m(x_3,m(1,x_7))$. The axioms of group theory are the following sentences:
$$\forall x\ \forall y\ \forall z\ (m(x,m(y,z))=m(m(x,y),z))\qquad (\text{associativity})$$
$$\forall x\ (m(x,1)=x\land m(1,x)=x)\qquad (\text{identity element})$$
$$\forall x\ \exists y\ (m(x,y)=1\land m(y,x)=1)\qquad (\text{inverse})$$
Groups are $L$-structures that are models of these sentences. Naturally, one replaces $m(x,y)$ with $x\cdot y$ for convenience. These axioms are in PNF. The first two are $\forall_1$ formulas and the third is a $\forall_2$ formula.

\subsection{Descriptive set theory}\label{ss.desc_set_theory}

In this section we present well-known theorems and notions. All of them can be found in \cite{KECHRIS}.

A topological space $(X,\tau)$ is called \textbf{Polish} if it is separable and completely metrizable (that is, there exists a complete metric $d$ on $X$ such that $(X,d)$ is homeomorphic to $(X,\tau)$).

It is clear that countable discrete spaces are Polish. In particular, $2=\{0,1\}$ and $\nat$ with the discrete topologies are Polish. The following proposition is well-known:

\begin{prop}\label{p.prod_polish}
Countable products of Polish spaces are Polish. In particular, $2^A$ and $\nat^A$ are Polish for any countable set $A$.
\end{prop}

A subset $E$ of a topological space $X$ is
\begin{itemize}
    \item \textbf{nowhere dense} if the closure of $E$ has empty interior,
    \item \textbf{meager} if it is a countable union of nowhere dense sets,
    \item \textbf{comeager} if its complement is meager.
\end{itemize}

A topological space $X$ is \textbf{Baire} if every nonempty open set is nonmeager in $X$.

Recall the Baire category theorem.

\begin{theorem}\label{t.bct}
Every completely metrizable topological space is Baire.
\end{theorem}

Since Polish spaces are completely metrizable, we may apply the Baire category theorem in them.

Recall that a subset $E$ of a topological space $X$ is $G_\delta$ if it is a countable intersection of open sets. Another well-known theorem characterizes Polish subspaces of Polish spaces.

\begin{theorem}\label{t.G_delta}
A subspace $E$ of a Polish space $X$ is Polish if and only if $E$ is $G_\delta$ in $X$.
\end{theorem}

It is a nice exercise to prove the following important proposition.

\begin{prop}\label{p.comeager}
A subset of a Baire space is comeager if and only if it contains a dense $G_\delta$ set.
\end{prop}

\begin{obs}\label{o.F_sigma}
It follows immediately from Proposition~\ref{p.comeager} that an $F_\sigma$ subset of a Baire space is nonmeager if and only if it has nonempty interior.
\end{obs}

A subset $E$ of a topological space $X$ has the \textbf{Baire property} (BP) if it can be written as $U\Delta M$ with $U$ open and $M$ meager in $X$ (where $\Delta$ denotes the symmetric difference). It is well-known that sets of the BP form a $\sigma$-algebra and therefore every Borel set has the BP.

A subset $A$ of a Polish space $Y$ is \textbf{analytic} if there exists a Polish space $X$, a Borel set $B\subseteq X$ and a continuous function $f:X\to Y$ such that $A=f(B)$. In a Polish space clearly every Borel set is analytic. It is also known that every analytic set has the BP.

A \textbf{Polish group} is a topological group whose topology is Polish. A well-known Polish group is $S_\infty$, that is, the group of all permutations of $\nat$ with the subspace topology inherited from $\nat^\nat$, see \cite[Example 7 of Subsection 9.B]{KECHRIS}.

Let $G$ be a topological group acting on a topological space $X$. The action is \textbf{continuous (resp. Borel)} if it is continuous (resp. Borel) as a $G\times X\to X$ function that maps $(g,x)$ to $g.x$. A subset $E$ of $X$ is $G$\textbf{-invariant} if it is a union of orbits. It is easy to see that if the action is continuous, then the map $x\mapsto g.x$ is a homeomorphism of $X$ for every $g\in G$. Thus a continuous action induces a homomorphism $\varphi:G\to\homeo(X)$, where $\homeo(X)$ is the group of homeomorphisms of $X$, hence $\varphi(G)$ is a subgroup of $\homeo(X)$.

We will need the following theorems \cite[Theorem 8.46]{KECHRIS} and \cite[Theorem 15.14]{KECHRIS}:

\begin{theorem}[Topological 0-1 law]\label{t.top0-1}
Let $X$ be a Baire space and $G$ a group of homeomorphisms of $X$ with the following homogeneity property: If $U,V$ are nonempty open sets in $X$, then there is $g\in G$ such that $g(U)\cap V\neq\emptyset$. Then every $G$-invariant subset of $X$ with the BP is either meager or comeager.
\end{theorem}

\begin{theorem}\label{t.isom_borel}
(Miller) If $(g,x)\mapsto g.x$ is a Borel action of the Polish group $G$ on the Polish space $X$, then every orbit is Borel.
\end{theorem}

\section{The space of multiplication tables}\label{s.the_space}

For convenience let $\nat\defeq\{1,2,3,\ldots\}$ throughout this paper.

By Proposition~\ref{p.prod_polish}  the space $\nnn$ of infinite tables of natural numbers is Polish. A clopen basis for $\nnn$ consists of sets of the form
\begin{equation}\label{eq.basis1}
    \left\{A\in\nnn:\ A(n_1,m_1)=k_1,\ldots,A(n_l,m_l)=k_l\right\}
\end{equation}
for $n_i,m_i,k_i\in\nat$, ($i=1,\ldots,l$). We will study the subspace
$$\calg\defeq\left\{A\in\nnn:\ A\text{ is the multiplication table of a group and $1$ is its identity element}\right\}.$$
We would like to apply the Baire category theorem (Theorem~\ref{t.bct}) in $\calg$. Therefore, by Theorem~\ref{t.G_delta} we need to prove that $\calg$ is $G_\delta$ in $\nnn$. We present the calculations to help readers not practiced in descriptive set theory. 

\begin{prop}
The set $\calg$ is $G_\delta$ in $\nnn$. Consequently, $\calg$ is a Polish space, hence the Baire category theorem holds in $\calg$.
\end{prop}

\begin{proof}
The set $\calg$ is the intersection of the following sets defined by the group axioms:
$$\calg_1=\left\{A\in\nnn:\ \forall n,m,k\in\nat\   A(A(n,m),k)=A(n,A(m,k))\right\}\qquad (\text{associativity}),$$
$$\calg_2=\left\{A\in\nnn:\ \forall n\in\nat\ A(n,1)=A(1,n)=n\right\}\qquad (1\text{ is the identity element}),$$
$$\calg_3=\left\{A\in\nnn:\ \forall n\in\nat\ \exists k\in\nat\ A(n,k)=A(k,n)=1\right\}\qquad (\text{inverses exist}).$$
Clearly
$$\calg_2=\bigcap_{n\in\nat}\underbrace{\left(\left\{A\in\nnn:\ A(n,1)=n\right\}\cap\left\{A\in\nnn:\ A(1,n)=n\right\}\right)}_{\text{clopen}}$$
is closed and
$$\calg_3=\bigcap_{n\in\nat}\bigcup_{k\in\nat}\left(\left\{A\in\nnn:\ A(n,k)=1\right\}\cap\left\{A\in\nnn:\ A(k,n)=1\right\}\right)$$
is $G_\delta$. For $\calg_1$ note that for any fixed triple $n,m,k\in\nat$ we have $A(A(n,m),k)=A(n,A(m,k))$ if and only if $\forall x,y,z\in\nat\ (A(n,m)=x\land A(m,k)=y\implies (A(x,k)=z \iff A(n,y)=z))$. Hence we may write $\calg_1$ as
$$\bigcap_{n,m,k\in\nat}\ \bigcap_{x,y,z\in\nat}\underbrace{\left(
\begin{array}{cc}
    & \{A:\ A(n,m)\neq x\}\cup\{A:\ A(m,k)\neq y\}\ \cup\\
    & \cup\ \{A:\ A(x,k)=z\land A(n,y)=z\}\cup\{A:\ A(x,k)\neq z\land A(n,y)\neq z\}
\end{array}
\right)}_{\text{clopen}}.$$
So $\calg_1$ is also closed. We conclude that $\calg=\calg_1\cap\calg_2\cap\calg_3$ is $G_\delta$.
\end{proof}

We reserve the notation $\cl G$ for the unique group of multiplication table $G$ and underlying set $\nat$.

\begin{remark}
When we consider elements of $\calg$ we use the usual shorthands of group theory. For example, to define the subspace of torsion groups we write
$$\{G\in\calg:\ \forall n\in\nat\ \exists k\in\nat\ n^k=1\}$$
instead of the rather cumbersome
$$\{G\in\calg:\ \forall n\in\nat\ \exists k\in\nat\ \underbrace {G(G(\ldots G(G(n,n),n),\ldots, n),n)}_{k\text{ times}}=1\}.$$
\end{remark}
\begin{remark}\label{r.inverse}
We also use inverses for convenience:

Fix any $n,m,k\in\nat$. Then e.g. $nm^{-1}=k$ abbreviates $\forall x\ (mx=1\implies nx=k)$. At first sight, this defines a closed subset of $\calg$. However, the group axioms imply that it is equivalent to $\exists x\ (mx=1\land nx=k)$, which defines an open set. Thus $nm^{-1}=k$ defines a clopen set.
\end{remark} 

\begin{remark}\label{r.basis1}
The above-mentioned standard clopen basis (\ref{eq.basis1}) of $\nnn$ induces a clopen basis for $\calg$: sets of the form
\begin{equation}\label{e.basis1}
    \{G\in\calg:\ \forall i\leq l\ (n_i\cdot m_i=k_i)\}
\end{equation}
with $l\in\nat$ and $n_i,m_i,k_i\in\nat$ for all $i\leq l$ constitute a basis.
For practical reasons we introduce another clopen basis for $\calg$.
\end{remark}

\begin{prop}\label{p.basis2}
For any finite sets $\{U_1,\dots,U_k\}$ and $\{V_1,\dots,V_l\}$ of words in variables $x_1,\dots,x_n$ and for any $a_1,\dots,a_n,b_1,\dots,b_k,c_1,\dots,c_l\in\nat$ the set
\begin{equation}
    \left\{G\in\calg:\ \bigwedge_{i=1}^k U_{i}(a_1,\dots,a_n)=b_i\land \bigwedge_{j=1}^l V_{j}(a_1,\dots,a_n)\neq c_j\right\}
\end{equation}
is clopen, and sets of this form constitute a basis for $\calg$.
\end{prop}

\begin{proof}
This family extends (\ref{e.basis1}), so it suffices to show that its elements are clopen. Since clopen sets form an algebra, it suffices to consider only one word $W$ and only the case of equation. The proof is a straightforward induction on the length of $W$ and uses the same argument as Remark~\ref{r.inverse}.
\end{proof}

\begin{obs}\label{o.induced_homeo}
An important observation is that permutations induce homeomorphisms. Let $\varphi:\nat\to\nat$ be a bijection that fixes 1. Then the \emph{induced homeomorphism} $h_\varphi:\calg\to\calg$ is defined as follows. Intuitively, we define $h_\varphi(G)$ by pushing forward the structure of $G$ via $\varphi$. More precisely, for any  $G\in\calg$ the multiplication table $h_\varphi(G)$ is defined by the equations $i\cdot j\defeq\varphi(G(\varphi^{-1}(i),\varphi^{-1}(j)))$ for all $i,j\in\nat$. Thus $\varphi$ is an isomorphism between $\cl G$ and $\cl{h_\varphi(G)}$. It is an easy exercise to verify that $h_\varphi$ is indeed a homeomorphism.
\end{obs}

\begin{remark}\label{r.polish_acts}
Let $S_\infty^*$ denote the set of permutations of $\nat$ that fix $1$. This is a clopen subgroup of the Polish group $S_\infty$, which was defined in Subsection~\ref{ss.desc_set_theory}. It is easy to check that Observation~\ref{o.induced_homeo} defines a continuous action of $S_\infty^*$ on $\calg$ whose orbits are exactly the isomorphism classes of $\calg$ (cf. \cite[page 96]{KECHRIS}). By Theorem~\ref{t.isom_borel} it follows that isomorphism classes are Borel.
\end{remark}

\subsection*{Sketch of the general setting}

Whenever we make a remark with the label ‘‘Model-theoretic generalization’’ we refer to the following setting.

Let $L$ be the first-order language with the following alphabet: $r_i$ is an $n_i$-ary relation symbol for every $i\in I$, $f_j$ is an $m_j$-ary function symbol for every $j\in J$, where $I$, $J$ are countable sets.

We define the space of $L$-structures on the universe $\nat$:
$$X_L\defeq\left(\prod_{i\in I} 2^{\nat^{n_i}}\right)\times \left(\prod_{j\in J} \nat^{\nat^{m_j}}\right).$$
It is a Polish space by Proposition~\ref{p.prod_polish}. An element $(x,y)\in X_L$ is a pair of sequences such that $x_i$ is a function from $\nat^{n_i}$ to $\{0,1\}$ for every $i\in I$ and $y_j$ is a function from $\nat^{m_j}$ to $\nat$ for every $j\in J$. Let $M(x,y)$ be the associated structure, that is, on the universe $\nat$ we interpret $r_i$ as the relation defined by $x_i$ and $f_j$ as the function $y_j$.

The generalization of Remark~\ref{r.basis1} is that a clopen basis element can be defined by prescribing finitely many functions and relations on finitely many tuples. Also sets defined by quantifier-free formulas constitute a clopen basis, which generalizes Proposition~\ref{p.basis2}

The Polish group $S_\infty$ of all permutations of $\nat$ acts on $X_L$ as follows. For any permutation $g\in S_\infty$ we let $g.(x,y)=(u,v)$ if and only if $u_i$ is the pushforward of $x_i$ via $g$ and $v_j$ is the pushforward of $y_j$ via $g$ for every $i\in I$ and $j\in J$. Observe that $M(x,y)$ and $M(u,v)$ are isomorphic if and only if there is some $g\in S_\infty$ such that $g.(x,y)=(u,v)$. Also note that for any $g\in G$ the map $x\mapsto g.x$ is a homeomorphism of $X_L$.

Let $\Gamma$ be an inductive theory, that is, a set of $\forall_2$-sentences of $L$. The theory $\Gamma$ is countable because $L$ is countable. It is straightforward to verify that
$$X_\Gamma\defeq\{(x,y)\in X_L:\ M(x,y)\models\Gamma\}$$
is a $G_\delta$ subset of $X_L$; therefore it is a Polish subspace. Clearly $X_\Gamma$ is isomorphism-invariant, hence for any $g\in S_\infty$ the map $x\mapsto g.x$ is a homeomorphism of $X_\Gamma$ as well. That is, we generalized Observation~\ref{o.induced_homeo}.

\section{Basic notions and denseness}\label{s.dense}

\begin{defi}\label{d.group_property}
A set $\calp\subseteq\calg$ is a \textbf{group property} if it is invariant under isomorphism, that is, for any $G\in\calg$ if $\cl G$ is isomorphic to $\cl H$ for some $H\in\calp$, then $G\in\calp$.
\end{defi}

Note that $\calp$ is a group property if and only if it is a union of orbits of $S_\infty^*$. One usually studies properties defined by (not necessarily first-order) formulas.

\begin{defi}
A group property $\calp\subseteq\calg$ is \textbf{generic} if it is a comeager subset of $\calg$.
\end{defi}

We often formulate the genericity of $\calp$ less formally: the generic $G\in\calg$ is of property $\calp$.
The main objective of this paper is to study generic group properties in $\calg$.

Proposition~\ref{p.comeager} gives us a standard way to prove that a group property $\calp$ is comeager: it suffices to show that it is $G_\delta$ and dense. In many cases it can be calculated directly from the definition that $\calp$ is $G_\delta$. Denseness can be shown separately, for example, by the following lemma.

\begin{lemma}\label{l.dense}
Let $\cals,\calp\subseteq\calg$ be group properties, $\calp\subseteq\cals$. If for every $G\in\cals$ there is some $H\in\calp$ such that $\cl G$ can be embedded into $\cl H$, then $\calp$ is dense in $\cals$.
\end{lemma}

\begin{remark}
The formulation of Lemma~\ref{l.dense} involves $\cals$ only to facilitate its direct application in subspaces. We mostly work with $\cals=\calg$.
\end{remark}

\begin{proof}[Proof of Lemma~\ref{l.dense}]
If $\cals$ is empty, then we are done. Otherwise, fix an arbitrary basic clopen $\calb=\{G\in\calg:\ \forall i,j\leq k\ (i\cdot j=m_{i,j})\}$ with $k\in\nat$ and $m_{i,j}\in\nat$ for all $i,j\leq k$ such that $\calb\cap\cals\neq\emptyset$. Pick some $G\in \calb\cap\cals$. Let $M\defeq\{1,\dots,k\}\cup\{m_{i,j}:\ i,j\leq k\}$.

By assumption there is some $H\in\calp$ and an embedding $\varphi:\cl G\to \cl H$. Since $\nat\setminus M$ and $\nat\setminus\varphi(M)$ are of the same cardinality, there is a bijection $\psi:\nat\to \nat$ extending $\varphi^{-1}|_{\varphi(M)}$. Consider the image of $H$ by the induced homeomorphism $h_\psi$. Now by the definition of $h_\psi$ we have $h_\psi(H)\in\calb$ because $\psi\supseteq\varphi^{-1}|_{\varphi(M)}$ and $\varphi$ is a homomorphism. Also $h_\psi(H)\in\calp$ since $\calp$ is a group property.
\end{proof}

\begin{remark}
\textbf{(Model-theoretic generalization)} Lemma~\ref{l.dense} can be generalized without any difficulty or additional assumption.
\end{remark}

As a direct application we show that simpleness is generic. This also follows from Theorem~\ref{t.alg_closed} by Remark~\ref{r.prop_of_ac}.

\begin{cor}\label{c.simple}
The generic countably infinite group is simple. That is, the set $\cals\subseteq\calg$ of multiplication tables of simple groups is comeager.
\end{cor}

\begin{proof}
It is known that every group can be embedded into a simple group of the same cardinality. (See the remark after \cite[Theorem 4.4]{KEGEL}.) Thus by Lemma~\ref{l.dense} it suffices to prove that $\cals$ is $G_\delta$ in $\calg$.

For a $G\in\calg$ the group $\cl G$ is simple if and only if for every $n\in\nat\setminus\{1\}$ the normal subgroup generated by $n$ in $\cl G$ contains every $k\in\nat$. Equivalently, for every $n\in\nat\setminus\{1\}$ every $k\in\nat$ equals to a word of finitely many conjugates of $n$. Thus
$$\cals=\left\{G\in\calg:\ \forall n\neq 1\ \forall k\ \exists l\ \exists g_1,\dots,g_l\ \exists W(x_1,\dots,x_l) \text{ such that }W(g_1^{-1}ng_1,\dots,g_l^{-1}ng_l)=k\right\}=$$
$$=\bigcap_{n,k\in\nat}\ \bigcup_{l\in\nat}\ \bigcup_{g_1,\dots,g_l\in\nat}\ \bigcup_{\substack{W\text{ is a word}\\ \text{in }l\text{ variables}}}\ \underbrace{\{G\in\calg:\ W(g_1^{-1}ng_1,\dots,g_l^{-1}ng_l)=k\}}_{\text{clopen by Proposition~\ref{p.basis2}}},$$
which is a $G_\delta$ set.
\end{proof}

\section{A 0-1 law for group properties}\label{s.0-1}

In this section we prove a 0-1 law for group properties and we present corollaries on isomorphism classes, embeddability and properties defined by first-order sentences.

\begin{notation}\label{n.supp}
If $\calb$ is a basic clopen set of the form
$\{G\in\calg:\ \forall i,j\leq k\ (i\cdot j=m_{i,j})\}$ with $k\in\nat$ and $m_{i,j}\in\nat$ for all $i,j\leq k$, then let $\supp\calb\defeq\{1,\dots,k\}\cup\{m_{i,j}:\ i,j\leq k\}$.
\end{notation}

By Observation~\ref{o.induced_homeo} and Remark~\ref{r.polish_acts} the Polish group $S_\infty^*$ acts continuously on $\calg$. As we have noted in Subsection~\ref{ss.desc_set_theory}, a continuous action of a group $G$ on a topological space $X$ induces a homomorphism from $G$ to $\homeo(X)$. In the case of $S_\infty^*$ and $\calg$ this homomorphism is $\varphi\mapsto h_\varphi$, which is injective because every nonidentity permutation $\varphi$ modifies some multiplication table. Thus we can identify $S_\infty^*$ with its image in $\homeo(\calg)$.

We need the following lemma:

\begin{lemma}\label{l.homogeneity}
For any nonempty basic clopen sets $\calu,\calv\subseteq\calg$ there is $\varphi\in S_\infty^*$ such that $h_\varphi(\calu)\cap \calv\neq\emptyset$.
\end{lemma}

\begin{proof}
Let $\calu=\{G\in\calg:\ \forall i,j\leq k\ (i\cdot j=m_{i,j}\}$. Let $M\defeq\supp \calu$ and $N\defeq\supp\calv$. Pick any $U\in \calu$ and $V\in \calv$. Let $\varphi:\nat\to\nat$ be a bijection such that $\varphi(1)=1$ and $\varphi(M)\cap N=\{1\}$. We claim that $h_\varphi(\calu)\cap\calv\neq\emptyset$. To verify this fix any bijection $\Psi:\nat\to\nat\times\nat$ that extends the finite maps
$$\begin{array}{lll}
i_M:\varphi(M)\to M\times\{1\},\ \varphi(m)\mapsto (m,1) & \text{and} & i_N:N\to \{1\}\times N,\ n\mapsto(1,n).
\end{array}$$
We define the multiplication table $H$ on $\nat$ by pulling back $U\times V$ via $\Psi$. Then we have $H\in \calv$ by $\Psi\supseteq i_N$. For $H\in h_\varphi(\calu)$ observe that
$$h_\varphi(\calu)=\{G\in\calg:\ \forall i,j\leq k\ (\varphi(i)\cdot\varphi(j)=\varphi(m_{i,j}))\}$$
and $\Psi\supseteq i_M$. Thus $h_\varphi(\calu)\cap\calv$ is nonempty.
\end{proof}

\begin{theorem}\label{t.0-1}
\textbf{(0-1 law)} Any group property $\calp\subseteq\calg$ that has the BP is either meager or comeager.
\end{theorem}

\begin{proof}
We apply Theorem~\ref{t.top0-1} to the Polish space $\calg$ and the Polish group $S_\infty^*$ (as a subgroup of $\homeo(\calg)$). The $S_\infty^*$-invariant sets are exactly the group properties and the required homogenity property holds by Lemma~\ref{l.homogeneity}.
\end{proof}

\begin{remark}\label{r.gen_0-1_JEP}
\textbf{(Model-theoretic generalization)} To generalize Lemma~\ref{l.homogeneity} and Theorem~\ref{t.0-1} it suffices to assume that the class of countably infinite models of $\Gamma$ has the Joint Embedding Property (JEP), that is, any two countably infinite models of $\Gamma$ can be embedded into a third one. As we have seen, direct products witness that the JEP holds for countably infinite groups.
\end{remark}

As we will illustrate by numerous examples, most group properties that one uses in practice possess the BP.

\begin{defi}
For a group $H$ let
$$\begin{array}{ccc}
    \cali_H\defeq\{G\in\calg:\ \cl G\cong H\} &\text{ and } & \cale_H\defeq\{G\in\calg:\ H\text{ is embeddable into }\cl G\}. \\
\end{array}$$
These are group properties.
For a first-order formula $\varphi$ of $L$ (the language of group theory) and an evaluation $e:\{x_1,x_2,\dots\}\to\nat$ of the variables of $L$ let $$\cals_{\varphi[e]}\defeq\{G\in\calg:\ \cl G\models \varphi[e]\}.$$
For a sentence $\varphi$ this is clearly a group property (isomorphic groups satisfy the same first-order sentences).
\end{defi}

The isomorphism classes are Borel and thereby they have the BP as we have already noted in Remark~\ref{r.polish_acts}.

\begin{theorem}\label{t.isom0-1}
Every isomorphism class is either meager or comeager. That is, for any group $H$ the set $\cali_H$ is either meager or comeager in $\calg$.
\end{theorem}

\begin{theorem}\label{t.embed0-1}
For any group $H$ the set $\cale_H$ is analytic; therefore it is either meager or comeager.
\end{theorem}

\begin{proof}
Fix a group $H$. We may assume that $
H$ is countable (possibly finite) because otherwise $\cale_H$ is empty. We find a $G_\delta$ subset $\cale'$ of the Polish space $\calg\times\nat^H$ such that its projection on $\calg$ is $\cale_H$.

For a $G\in\calg$ the group $H$ is embeddable into $\cl G$ if and only if there is an injection $\Psi:H\to\cl G$ such that for every $a,b,c\in H$ for which $H\models a\cdot b=c$ we have $\cl G\models\Psi(a)\cdot\Psi(b)=\Psi(c)$. Thus $\cale_H$ is the projection of the set $$\begin{array}{cccc}
    \ds\cale'\defeq(\calg\times\calb)\cap\left(\bigcap_{\substack{a,b,c\in H\\ H\models\ a\cdot b=c}}\calh(a,b,c)\right), &\text{ where } & \calb=\left\{\Psi\in\nat^H:\ \Psi\text{ is a injection}\right\} &\text{ and} \\
\end{array}$$
$$\calh(a,b,c)=\left\{(G,\Psi)\in\calg\times\nat^H:\ \cl G\models\Psi(a)\cdot\Psi(b)=\Psi(c)\right\}.$$

It is easy to verify that $\calb$ is closed. Also note that $$\calh(a,b,c)=\bigcup_{n,k,l\in\nat}\left(\begin{array}{l}
    \{(G,\Psi)\in\calg\times\nat^H:\ \cl G\models n\cdot k=l\}\cap  \\
    \{(G,\Psi)\in\calg\times \nat^H:\ \Psi(a)=n,\ \Psi(b)=k,\ \Psi(c)=l\}
\end{array}\right)$$
is open. Hence $\cale'$ is $G_\delta$ and $\cale_H$ is analytic.
\end{proof}

Compare the following theorem to \cite[Proposition 16.7]{KECHRIS}.

\begin{theorem}\label{t.formula_borel}
For every first-order formula $\varphi$ of $L$ and every evaluation $e:\{x_1,x_2,\dots\}\to\nat$ of the variables of $L$ the set $\cals_{\varphi[e]}$ is Borel in $\calg$. Thus for sentences, which define isomorphism-invariant sets, it is either meager or comeager.
\end{theorem}

\begin{proof}
We proceed by induction on the complexity of $\varphi$.  Fix any evaluation $e$ of the variables $x_1,x_2,\dots$.

\textbf{Case 1.} The formula $\varphi$ is atomic. Then $\varphi[e]$ is of the form $a_1\cdot\ldots\cdot a_n=b_1\cdot\ldots\cdot b_m$ where $a_i,b_j\in\nat$ for all $1\leq i\leq n$ and $1\leq j\leq m$. Hence $\cals_{\varphi[e]}$ is clopen by Proposition~\ref{p.basis2}.

\textbf{Case 2.} The formula $\varphi$ is of the form $\lnot\psi$. Then $\cals_{\varphi[e]}$ is the complement of $\cals_{\psi[e]}$, hence it is Borel.

\textbf{Case 3.} The formula $\varphi$ is of the form $\psi_1\land\psi_2$. Then $\cals_{\varphi[e]}=\cals_{\psi_1[e]}\cap\cals_{\psi_2[e]}$ is Borel.

\textbf{Case 4.} The formula $\varphi$ is of the form $\exists x_i\ \psi(x_i)$. Then $\cl G\models \varphi[e]$  if and only if for some $a\in\nat$ we have $\cl G\models\psi[e(x_i/a)]$, where $e(x_i/a)$ is the evaluation that maps $x_i$ to $a$ and coincides with $e$ on other variables. Thus
$$\cals_{\varphi[e]}=\bigcup_{a\in\nat}\cals_{\psi[e(x_i/a)]}$$
is Borel.
\end{proof}

\begin{cor}\label{c.elem_equiv}
There is a comeager elementary equivalence class in $\calg$. That is, there exists a comeager set $\cals\subseteq\calg$ such that for every $G\in\cals$ and $H\in\calg$ the groups $\cl G$ and $\cl H$ are elementarily equivalent (as $L$-structures) if and only if $H\in\cals$ .
\end{cor}

\begin{proof}
Let $\Gamma$ be the set of sentences $\varphi$ of $L$ such that $\cals_\varphi$ is comeager. Then
$$\cals\defeq\bigcap_{\varphi\in\Gamma}\cals_\varphi$$
is comeager. By Theorem~\ref{t.formula_borel} for every sentence $\varphi$ of $L$ exactly one of $\varphi$ and $\lnot\varphi$ is in $\Gamma$. Thus for any $G\in\cals$ and $H\in\calg$ the groups $\cl G$ and $\cl H$ are elementarily equivalent if and only if $H\in\cals$.
\end{proof}

\begin{remark}\label{r.general_0-1}
\textbf{(Model-theoretic generalization)} It is straightfoward to generalize each of Theorem~\ref{t.isom0-1}, Theorem~\ref{t.embed0-1}, Theorem~\ref{t.formula_borel} and Corollary~\ref{c.elem_equiv}.
\end{remark}

\section{Algebraically closed groups}\label{s.alg_closed}

We wish to study isomorphism and embeddability more closely. For this purpose we need algebraically closed groups. 

\begin{theorem}\label{t.alg_closed}
The set $\calc\defeq\{G\in\calg:\ \cl G\text{ is algebraically closed}\}$ is comeager in $\calg$.
\end{theorem}

\begin{proof}
Let $F$ be the free group generated by $X=\{x_1,x_2\dots\}$. Let $G\in\calg$. By definition, $\cl G$ is algebraically closed if and only if the following holds. For every finite system $(E,I)$ of equations and inequations with elements from $F\star\cl G$ either $(E,I)$ is inconsistent with $\cl G$ or it has a solution in $\cl G$ (see Subsection~\ref{ss.algebra} for the definitions). Note that elements of $F\star \cl G$ are words with letters from $X$ and $\nat$. Clearly, there are countably many such words. Thus it suffices to prove that for any finite system $(E,I)$ the set
$$\calc(E,I)\defeq\{G\in\calg:\ (E,I)\text{ is inconsistent with } \cl G\text{ or $(E,I)$ has a solution in } \cl G\}$$
contains a dense open set. Fix

\begin{tabular}{lll}
$E=\{U_1(x_1,\dots,x_n),\dots,U_k(x_1,\dots,x_n)\}$ & \text{ and } & $I=\{V_1(x_1,\dots,x_n),\dots,V_l(x_1,\dots,x_n)\}$ \\
\end{tabular}

\noindent where $x_1,\dots,x_n$ are the variables occuring in some element of $E$ or $I$.

Fix a nonempty basic clopen set $\calb=\{G\in\calg:\ \forall i,j\leq k\ (i\cdot j=m_{i,j})\}$ with $k\in\nat$ and $m_{i,j}\in\nat$ for all $i,j\leq k$. We need to show that $\calb\cap\calc(E,I)$ contains a nonempty clopen set. There are two cases.

\textbf{Case 1.} There exists $H\in\calb$ such that $(E,I)$ has a solution in $\cl H$. That is, for some natural numbers $a_1,\dots,a_n$ we have $\bigwedge_{i=1}^k U_i(a_1,\dots,a_n)=1$ and $\bigwedge_{j=1}^l V_j(a_1,\dots,a_n)\neq 1$ in $\cl H$. Then
$$\calu\defeq\left\{G\in\calg:\ \bigwedge_{i=1}^k U_i(a_1,\dots,a_n)=1\land \bigwedge_{j=1}^l V_j(a_1,\dots,a_n)\neq 1\right\}$$
is clopen in $\calg$ by Proposition~\ref{p.basis2}. Now $\calu\subseteq\calc(E,I)$ and $H\in\calu\cap\calb$, which completes Case 1.

\textbf{Case 2.} For every $G\in\calb$ the finite system $(E,I)$ is unsolvable in $\cl G$. It suffices to prove that for every $G\in\calb$ the finite system $(E,I)$ is inconsistent with $\cl G$.

Suppose that there is some $H\in\calb$ and a group $\wtilde K\geq \cl H$ such that $(E,I)$ has a solution $a_1,\dots,a_n$ in $\wtilde K$. Clearly, we may assume that $\wtilde K$ is countable since otherwise we could replace it by one of its countably generated subgroups. Let $M$ be the set of natural numbers occuring in elements of $E$ or $I$. We choose any bijection $\varphi:\wtilde K\to \nat$ that extends the identity of the finite set $M\cup\supp\calb$ (recall Notation~\ref{n.supp}). We define the multiplication table $K$ on $\nat$ by pushing forward the structure of $\wtilde K$ via $\varphi$. Then $\cl K\in\calb$ and $\varphi(a_1),\dots,\varphi(a_n)$ is a solution of $(E,I)$ in $\cl K$, a contradiction.
\end{proof}

\begin{remark}
\textbf{(Model-theoretic generalization)} Theorem~\ref{t.alg_closed} can be generalized without any difficulty or additional assumption. As we have mentioned in Subsection~\ref{ss.logic}, the natural generalization of algebraic closedness is existential closedness. We remark that Pouzet and Roux \cite{POUZET} proved the genericity of existential closedness among countably infinite models of a fixed universal theory.
\end{remark}

\begin{cor}
The generic countably infinite group is simple, not finitely generated and not locally finite. That is, the set $\{G\in\calg:\ \cl G\text{ is simple, not fin. gen. and not loc. fin.}\}$ is comeager in $\calg$.
\end{cor}

\begin{proof}
It follows immediately from Theorem~\ref{t.alg_closed} and Remark~\ref{r.prop_of_ac}.
\end{proof}

We present a lemma that will prove to be useful.

\begin{lemma}\label{l.sol_in_all_alg_c}
Let $\Phi$ be an existential sentence of $L$ (the language of group theory). If $\Phi$ holds in some group, then $\Phi$ holds in every algebraically closed group.
\end{lemma}

\begin{proof}
Let $H$ be a group such that $H\models\Phi$. By the definition of algebraically closed groups it suffices to show that for any algebraically closed group $G$ we have $G\times H\models\Phi$. Since $\{1\}\times H$ and $H$ are isomorphic, $\{1\}\times H\models \Phi$. Hence $G\times H\models \Phi$ because $\Phi$ is existential.
\end{proof}

\begin{theorem}\label{t.alg_isom_dense}
The isomorphism class of any countably infinite algebraically closed group is dense in $\calg$.
\end{theorem}

\begin{proof}
Fix a nonempty basic clopen $\calb=\{G\in\calg:\ \forall i,j\leq k\ (i\cdot j=m_{i,j})\}$ with $k\in\nat$ and $m_{i,j}\in\nat$ for all $i,j\leq k$. Fix any algebraically closed group $A\in\calg$. Pick any $H\in\calb$. We associate variables $x_i$ to $i$ for each $i\leq k$ and $x_{i,j}$ to $(i,j)$ for each $i,j\leq k$. For each $i,j,l\leq k$ let
$$
\varphi_{i,j,l}=
\begin{cases}
x_{i,j}=x_l\qquad\text {if}\quad m_{i,j}=l, \\
x_{i,j}\neq x_l\qquad\text{if}\quad m_{i,j}\neq l,
\end{cases}
$$
and for each $i,j,r,s\leq k$ let
$$
\varphi_{i,j,r,s}=
\begin{cases}
x_{i,j}=x_{r,s}\qquad\text {if}\quad m_{i,j}=m_{r,s}, \\
x_{i,j}\neq x_{r,s}\qquad\text{if}\quad m_{i,j}\neq m_{r,s}.
\end{cases}
$$
Now let $\Phi$ be the existential closure of the following formula:
$$\left(\bigwedge_{i,j\leq k} x_i\cdot x_j=x_{i,j}\right)\land\left(\bigwedge_{\substack{i,j\leq k\\ i\neq j}}x_i\neq x_j\right)\land\left(\bigwedge_{i,j,l\leq k}\varphi_{i,j,l}\right)\land\left(\bigwedge_{i,j,r,s\leq k}\varphi_{i,j,r,s}\right).$$
Clearly $H\models\Phi$. Thus $\cl A\models\Phi$ by Lemma~\ref{l.sol_in_all_alg_c}. That is, there are numbers $a_i, n_{i,j}\in \nat$ such that $\cl A\models a_i\cdot a_j=n_{i,j}$ for each $i,j\leq k$ and two of them equal if and only if the corresponding elements equal in $H$. Note that $a_1=1$ since both $a_1\cdot a_1=n_{1,1}$ and $n_{1,1}=a_1$ hold. Let $\alpha:\nat\to\nat$ be a bijection that extends the finite map $a_i\mapsto i$, $n_{i,j}\mapsto m_{i,j}$ for each $i,j\leq k$. Then $h_\alpha(A)\in\calb$ and $\overline{h_\alpha(A)}$ is isomorphic to $\cl A$ (where $h_\alpha$ is the induced homeomorphism). 
\end{proof}

\begin{remark}
\textbf{(Model-theoretic generalization)} Similarly to Remark~\ref{r.gen_0-1_JEP} it suffices to assume the JEP to generalize Lemma~\ref{l.sol_in_all_alg_c}. Then Theorem~\ref{t.alg_isom_dense} generalizes without further assumptions.
\end{remark}

\section{Embeddability and isomorphism classes}\label{s.embed_isom}

Recall that $\cale_H=\{G\in\calg:\ H\text{ is embeddable into }\cl G\}$.

\begin{defi}\label{d.gen_emb}
A group $H$ is \textbf{generically embeddable} if $\cale_H$ is comeager.
\end{defi}

In this section we characterize generically embeddable groups and we prove that every ismorphism class is meager in $\calg$. 

First we need two lemmas.

\begin{lemma}\label{l.fin_gen_extends}
Let $G$ and $H$ be countable groups such that $H$ is homogeneous. Then $G$ can be embedded into $H$ if and only if every finitely generated subgroup of $G$ can be embedded into $H$.
\end{lemma}

\begin{proof}
$\implies$: This is obvious.

$\impliedby$: Write $G$ as a union of finitely generated subgroups $\bigcup_{n=1}^\infty G_n$ such that $G_n\subseteq G_{n+1}$ for every $n$. It suffices to find a sequence $i_n:G_n\embeds H$ of embeddings such that $i_{n+1}\supseteq i_n$ for every $n$ because then $i\defeq\bigcup_{n=1}^\infty i_n$ embeds $G$ into $H$.

Suppose $i_n:G_n\embeds H$ is given. By assumption there is an embedding $j_{n+1}:G_{n+1}\embeds H$ not necessarily extending $i_n$. However, by homogeneity there is an automorphism $\alpha_{n+1}$ of $H$ that extends the isomorphism $(i_n\circ j_{n+1}^{-1})|_{j_{n+1}(G_n)}$. Therefore $i_{n+1}\defeq\alpha_{n+1}\circ j_{n+1}$ is an embedding of $G_{n+1}$ into $H$ that extends $i_n$.
\end{proof}

\begin{cor}\label{c.emb_iff_fin_gen}
A group $H$ is generically embeddable if and only if $H$ is countable and every finitely generated subgroup of $H$ is generically embeddable.
\end{cor}

\begin{proof}
$\implies$: This is clear.

$\impliedby$: Let $H_1,H_2,\dots$ be an enumeration of the finitely generated subgroups of $H$. Let $\calh=\{G\in\calg:\ \cl G\text{ is homogeneous}\}$. By Theorem~\ref{t.alg_closed} and Proposition~\ref{p.alg_closed_strongly_hom} the set $\calh$ is comeager. Since $\cale_{H_n}$ is comeager for each $n$, $(\bigcap_{n=1}^\infty\cale_{H_n})\cap\calh$ is also comeager. By Lemma~\ref{l.fin_gen_extends} we have $(\bigcap_{n=1}^\infty\cale_{H_n})\cap\calh\subseteq\cale_H$. Thus $H$ is generically embeddable.
\end{proof}

\begin{lemma}\label{l.F_sigma}
For any finitely generated group $H$ the set $\cale_H$ is $F_\sigma$.
\end{lemma}

\begin{proof}
Fix a finite generating set $\{a_1,\dots,a_n\}$ for $H$. Let $G\in\calg$ be arbitrary.

\textbf{Claim.} The group $H$ is embeddable into $\cl G$ if and only if there are $b_1,\dots,b_n\in\nat$ such that for every word $W$ in $n$ variables $W(a_1,\dots,a_n)=1$ in $H$ $\iff$ $W(b_1,\dots,b_n)=1$ in $\cl G$.

If there is an embedding $i:H\embeds \cl G$, then clearly for every word $W$ in $n$ variables $W(a_1,\dots,a_n)$ maps to $W(i(a_1),\dots,i(a_n))$, therefore $b_j=i(a_j)$ for each $1\leq j\leq n$ is suitable. On the other hand, if for some $b_1,\dots,b_n\in \cl G$ we have $W(a_1,\dots,a_n)=1$ in $H$ $\iff$ $W(b_1,\dots,b_n)=1$ in $\cl G$ for every word $W$ in $n$ variables, then $H$ and $\langle b_1,\dots,b_n\rangle_{\cl G}$ have the same presentation, hence they are isomorphic. This proves the claim.

Thus we may write $\cale_H$ as
$$\bigcup_{b_1,\dots,b_n\in\nat}\ \bigcap_{\substack{W\text{ is a word}\\ \text{in }n\text{ variables}}}\underbrace{\{G\in\calg:\ W(a_1,\dots,a_n)=1\text{ in }H \iff W(b_1,\dots,b_n)=1\text{ in }\cl G\}}_{\text{clopen by Proposition }\ref{p.basis2}},$$
which proves the lemma.
\end{proof}

Now we turn to the characterization.

\begin{theorem}\label{t.embed}
The following are equivalent:

(1) The group $H$ is generically embeddable.

(2) The group $H$ can be embedded into every algebraically closed group.

(3) The group $H$ is countable and every finitely generated subgroup of $H$ has solvable word problem.
\end{theorem}

\begin{proof}
(2)$\implies$(3) follows from Theorem~\ref{t.alg_c_word_p}.

(3)$\implies$(2): Suppose (3). By Theorem~\ref{t.alg_c_word_p} every finitely generated subgroup of $H$ can be embedded into every algebraically closed group. Since every algebraically closed group is homogeneous by Proposition~\ref{p.alg_closed_strongly_hom}, (2) follows from Lemma~\ref{l.fin_gen_extends}.

(2)$\implies$(1) is immediate from Theorem~\ref{t.alg_closed}.

(1)$\implies$(2): Fix a generically embeddable group $H$. By Lemma~\ref{l.fin_gen_extends} it suffices to prove that any finitely generated subgroup $K$ of $H$ can be embedded into every algebraically closed group. Fix $K$. Note that $\cale_K$ is comeager because it contains $\cale_H$. On the other hand, $\cale_K$ is $F_\sigma$ by Lemma~\ref{l.F_sigma}, hence it has nonempty interior by Observation~\ref{o.F_sigma}. However, isomorphism classes of algebraically closed groups are dense in $\calg$ by Theorem~\ref{t.alg_isom_dense}; therefore $K$ can be embedded into every algebraically closed group.
\end{proof}

Now we turn to isomorphism classes.

\begin{theorem}\label{t.isom}
Every isomorphism class is meager. That is, for every group $H$ the set $\cali_H$ is meager in $\calg$.
\end{theorem}

\begin{proof}
Suppose there is some group $H$ such that $\cali_H$ is nonmeager. Then $\cali_H$ is comeager by Theorem~\ref{t.isom0-1}. On the one hand, this implies that $H$ is algebraically closed by Theorem~\ref{t.alg_closed}. On the other hand, $H$ is generically embeddable, hence every finitely generated subgroup of $H$ has solvable word problem by Theorem~\ref{t.embed}.

However, it follows from a result of Miller \cite{MILLER}, as it is pointed out in \cite[page 58]{MACINTYRE1}, that every algebraically closed group has a finitely generated subgroup with unsolvable word problem. Consequently, such an $H$ cannot exist.
\end{proof}

\begin{remark}
\textbf{(Model-theoretic generalization)} There are two obstacles in the way of generalization:

(1) It is \textbf{not} true in general that countably infinite existentially closed models of a $\forall_2$-theory are always homogeneous. Note that Proposition~\ref{p.alg_closed_strongly_hom} is specific to group theory.

(2) We are not aware of any analogous theory of the word problem in this general model-theoretic setting.
\end{remark}

\section{Abelian groups}\label{s.abelian_groups}

In this section we prove that there is a generic countably infinite Abelian group. That is, there is a comeager isomorphism class in the subspace $\cala\defeq\{G\in\calg:\ \cl G\text{ is Abelian}\}$. To avoid confusion we \textbf{do not} switch to the additive notation.

First of all observe that
$$\cala=\{G\in\calg:\ \forall n,k\ (nkn^{-1}k^{-1}=1)\}=\bigcap_{n,k\in\nat}\{G\in\calg:\ nkn^{-1}k^{-1}=1\}$$
is a closed subspace of $\calg$, therefore it is a Polish space.

\begin{remark}
It is very easy to see that $\cala\subseteq\calg$ has empty interior, hence it is nowhere dense. Thus for a comeager property $\calp\subseteq\calg$ the set $\calp\cap\cala$ may be meager in $\cala$.
\end{remark}

The following theorem is well-known, see \cite[Theorem 3.1 of Chapter 4]{FUCHS}.

\begin{theorem}\label{t.divisible}
Every divisible group is of the form
\begin{equation}\label{e.divisible}
\left(\bigoplus_{p\in\mathbf{P}}\int[p^\infty]^{(I_p)}\right)\oplus\rat^{(I)},
\end{equation}
where $\mathbf{P}$ is the set of prime numbers, $\int[p^\infty]$ is the Prüfer $p$-group, $I$ and $I_p$ are arbitrary sets of indices, and for a group $G$ and a set $J$ the term $G^{(J)}$ abbreviates the direct sum $\bigoplus_{j\in J}G$.
\end{theorem}

Let $A\in\calg$ be such that
$$\cl A\cong\bigoplus_{p\in\mathbf{P}}\int[p^\infty]^{(\nat)}.$$
\begin{remark}\label{r.abelian_embeds}
Another well-known theorem is that every Abelian group can be embedded into a divisible group, see \cite[Theorem 1.4 of Chapter 4]{FUCHS}. Thus by Theorem~\ref{t.divisible} every countable Abelian torsion group can be embedded into $\cl A$.
\end{remark}

\begin{remark}\label{r.unique_abelian}
If $G$ is a divisible Abelian torsion group, then it can be written as the torsion summand in Theorem~\ref{t.divisible}. If every finite Abelian group can be embedded into $G$, then $I_p$ is infinite for every $p\in\mathbf{P}$ since $G$ must contain infinitely many elements of order $p$ for every $p$. If $G$ is countable, then $I_p$ is countable for every $p\in\mathbf{P}$. Therefore, $\cl A$ is the unique countable, divisible Abelian torsion group that contains every finite Abelian group (up to isomorphism).
\end{remark}

\begin{prop}\label{p.abel_G_delta}
The sets $\cald\defeq\{G\in\cala:\ \cl G\text{ is divisible}\}$, $\calt\defeq\{G\in\cala:\ \cl G\text{ is a torsion group}\}$ and $\calf\defeq\{G\in\cala:\ \text{every finite Abelian group can be embedded into }\cl G\}$ are $G_\delta$ in $\cala$.
\end{prop}

\begin{proof}
By definition
$$\cald=\{G\in\cala:\ \forall n,k\in\nat\ \exists m\in\nat\ (m^k=n)\}=\bigcap_{n,k\in\nat}\ \bigcup_{m\in\nat}\ \underbrace{\{G\in\cala:\ m^k=n\}}_{\text{clopen}},$$
which is $G_\delta$ in $\cala$.
Again, by definition
$$\calt=\{G\in\cala:\ \forall n\ \exists k\ (n^k=1)\}=\bigcap_{n\in\nat}\ \bigcup_{k\in\nat}\ \underbrace{\{G\in\cala:\ n^k=1\}}_{\text{clopen}},$$
which is $G_\delta$ in $\cala$.
For $\calf$ note that there are countably infinitely many finite Abelian groups up to isomorphism, hence it suffices to prove that for any fixed finite Abelian group $H$ the set $\cala\cap\cale_H$ is $G_\delta$. Let $\dom(H)=\{h_1,\dots,h_n\}$. For a $G\in\cala$ the group $H$ is embeddable into $\cl G$ if and only if there exist pairwise distinct numbers $g_1,\dots,g_n\in \nat$ such that $h_i\cdot h_j=h_l$ in $H$ implies $g_i\cdot g_j=g_l$ in $\cl G$ for all $i,j,l\leq n$. That is,
$$\calf=\bigcup_{\substack{g_1,\dots,g_n\in\nat\\ \text{p.~distinct}}}\ \bigcap_{\substack{h_i,h_j,h_l\in H,\\ h_i\cdot h_j=h_l}}\ \{G\in\cala:\ g_i\cdot g_j=g_l\},$$
which is open in $\cala$.
\end{proof}

From Remark~\ref{r.unique_abelian} we know that $\cali_{\cl A}=\cala\cap\cald\cap\calt\cap\calf$, which is $G_\delta$ in $\cala$. Now we show that it is dense in $\cala$.

\begin{prop}\label{p.abel_dense}
The isomorphism class $\cali_{\cl A}$ is dense in $\cala$.
\end{prop}

\begin{proof}
Fix any nonempty basic clopen $\calb=\{G\in\cala:\ \forall i,j\leq k\ (i\cdot j=m_{i,j})\}$. Pick some $H\in\calb$ and a divisible Abelian group $K$ with an embedding $\varphi:\cl H\embeds K$ (see Remark~\ref{r.abelian_embeds}). Clearly, $K$ can be chosen to be countable. We write $K$ in the form
$$\left(\bigoplus_{p\in\mathbf{P}}\int[p^\infty]^{(I_{K,p})}\right)\oplus\rat^{(I_K)}.$$
Consider the finite subset $\varphi(\supp\calb)\subseteq K$ (recall Notation~\ref{n.supp}). Every $x\in\varphi(\supp\calb)$ has finitely many nonzero coordinates $x_i$ with $i\in I_K$. Let $n$ be a natural number greater than $2\cdot\max\{|x_i|:\ x\in\varphi(\supp\calb), i\in I_K\}$. Let $N$ be the subgroup of $K$ generated by elements of the form $((0,0,\dots),(0,\dots,0,\underset{r\text{th}}{n},0,\dots))$ for every $r\in I_K$. Let $\psi:K\to K/N$ be the quotient map. Clearly $K/N$ is
$$\left(\bigoplus_{p\in\mathbf{P}}\int[p^\infty]^{(I_{K,p})}\right)\oplus(\rat/n\int)^{(I_K)}.$$
This is a countable Abelian torsion group, thus there is an embedding $\nu: K/N\embeds \cl A$ by Remark~\ref{r.abelian_embeds}. Notice that $\nu\circ\psi\circ\varphi:\ \cl H\to \cl A$ is homomorphism that is injective on the set $\supp\calb\subseteq\cl H$ by the choices of $n$ and $N$. Thus there is a bijection $\vartheta:\nat\to\nat$ that extends the finite bijection $(\nu\circ\psi\circ\varphi|_{\supp\calb})^{-1}$. For such a $\vartheta$ we have $h_\vartheta(A)\in\calb$ because $H$ and $h_\vartheta(A)$ coincide on $\{1,\dots,k\}\times\{1,\dots,k\}$. Since $\overline{h_\vartheta(A)}\cong\cl A$, we conclude that $\cali_{\cl A}$ is dense in $\cala$.
\end{proof}

\begin{cor}\label{c.generic_abelian}
The isomorphism class $\cali_{\cl A}$ is comeager in $\cala$.
\end{cor}

\subsection*{Related results}

Let us sketch some parts of the paper \cite{TENT1} of Z.~Kabluchko and K.~Tent. 

Let $L$ be a countable language. Let $\calc$ be a countably infinite set of finitely generated $L$-structures. The set $\calc$ is a \textbf{Fra\"issé class} if it has the hereditary property, the joint embedding property and the amalgamation property. See \cite{TENT1} for more definitions. The \textbf{age} of an $L$-structure $S$ is the set of isomorphism types of finitely generated substructures of $S$.

Let $\mathbb{S}$ be the set of all not finitely generated $L$-structures on $\nat$ with age contained in $\calc$. Consider the following topology $\calt$ on $\mathbb{S}$. For every finitely generated $L$-structure $B$ with $\dom(B)\subseteq\nat$ whose age is contained in $\calc$, let $\calo_B$ be the set of $S\in\mathbb{S}$ whose restriction to $\dom(B)$ coincides with $B$. Sets of the form $\calo_B$ constitute a basis for a topology $\calt$. The space $(\mathbb{S},\calt)$ is a Baire space.

\begin{theorem}\label{t.fraisse}
(Fra\"issé) For any Fra\"issé class $\calc$ there exists a countable structure $M$ with the following properties:

(1) The age of $M$ is $\calc$. (universality)

(2) Any isomorphism between finitely generated substructures of $M$ extends to an automorphism of $M$. (homogeneity)

Moreover, $M$ is unique up to isomorphism. We call $M$ the \textbf{Fra\"issé limit} of $\calc$.
\end{theorem}

Kabluchko and Tent proved in \cite{TENT1} that for a Fra\"issé class $\calc$ that does not contain its Fra\"issé limit the isomorphism class of the Fra\"issé limit of $\calc$ is comeager in the associated space $(\mathbb{S},\calt)$.

It is not hard to prove that finite Abelian groups form a Fra\"issé class $\mathfrak{F}$ whose Fra\"issé limit is $\cl A$. Note that the associated space $\mathbb{S_\mathfrak{F}}$ consists of countably infinite Abelian torsion groups. It can be shown that since $\mathfrak{F}$ consists of finite groups, the associated topology $\calt_\mathfrak{F}$ is the same as the subspace topology induced by the natural embedding $\mathbb{S}_\mathfrak{F}\embeds\calg$, see  \cite[page 5]{TENT1}. From this point of view Corollary~\ref{c.generic_abelian} says that the isomorphism type of $\cl A$ is generic not only in $\mathbb{S}$ but among all countably infinite Abelian groups.

One may consider the Fra\"issé class $\mathfrak{A}$ of all finitely generated Abelian groups. Its Fra\"issé limit is
$$B=\left(\bigoplus_{p\in\mathbf{P}}\int[p^\infty]^{(\nat)}\right)\oplus\rat^{(\nat)}.$$
The associated space $\mathbb{S}_\mathfrak{A}$ consist of all countably infinite not finitely generated Abelian groups. It is easy to show that in this case the associated topology $\calt_\mathfrak{A}$ is \textbf{not} the same as the subspace topology induced by $\mathbb{S}_\mathfrak{A}\embeds \calg$. The set $\mathbb{S}_\mathfrak{A}$ with any of these two topologies is a Baire space. In the first case, $B$ represents the comeager isomorphism class. In the second case, $\cl A$ does.

\section{Connections with infinite games}\label{s.games}

Infinite games are closely related to Baire category. See \cite[Subsection~20.A]{KECHRIS} for a general introduction to infinite games. Our results can be reformulated as theorems on a specific infinite game. Let us introduce this game.

There are two players, say Eve and Odd, Eve starts. At the beginning of the game they are given an empty infinite table (representing $\nat\times\nat$). They alternately choose finitely many cells and write a natural number in each cell. We want them to construct a group multiplication table in infinitely many steps, hence we add some extra rules. At the $n$th step the next player has to guarantee the following:

(1) After the $n$th step there exists a group multiplication table that extends their partially filled table.

(2) Cells corresponding to elements of $\{1,\dots,n+1\}\times\{1,\dots,n+1\}$ are filled in.

(3) The elements $1,\dots,n+1$ have inverses. That is, each of the first $n+1$ rows and columns contain a $1$.

As one checks easily, these rules guarantee that the result of a run of the game is a group multiplication table $G$. Let $\calp\subseteq\calg$ be a group property. We say that Odd wins if the resulting group $\cl G$ is of property $\calp$ (or more formally, if $G\in\calp$). Otherwise, Eve wins. We denote this game by $G(\calp)$.

In fact, $G(\calp)$ is a special case of the Banach-Mazur game (see \cite[Subsection~8.H]{KECHRIS}):

The players together define a decreasing sequence of basic clopen sets $\calb_0\supseteq\calb_1\supseteq\calb_2\supseteq\dots$ in $\calg$. Rule (1) assures that $\calb_i$ is nonempty and rules (2) and (3) assure that the intersection is a singleton. It is well-known that if the Banach-Mazur game is played on a Polish space, then Odd has a winning strategy if and only if his winning set (now $\calp\subseteq\calg$) is comeager \cite{KECHRIS}.

Thus a group property $\calp\subseteq\calg$ is generic if and only if Odd has a winning strategy in $G(\calp)$. It is a pleasant exercise to reformulate our results as theorems about $G(\calp)$.

As we have written in the introduction, infinite games are intensively studied from the model-theoretic point of view. Let us mention a recent example.

A.~Krawczyk and W.~Kubiś studied a variant $\bm(\calf,\cala)$ of the Banach-Mazur game in \cite{KUBIS1}. In their game the players alternately choose finitely generated structures $A_n$ from a given class $\calf$. Together they define an increasing chain $A_0\subseteq A_1\subseteq A_2\subseteq\dots$, which gives a limit structure. Among other results they characterized when Odd can attain a given isomorphism class.

It is not clear if there is a natural topology with which $\bm(\calf,\cala)$ is a special case of the general Banach-Mazur game. However, one may define genericity directly by winning strategies: a group property is generic if Odd can always attain it. With this natural definition $\bm(\calf,\cala)$ is \textbf{not equivalent} to our game $G(\calp)$. Indeed, observe that in $\bm(\calf,\cala)$ for any finitely generated group $H$ both players can attain the property of $H$ being embeddable in the resulting group. On the other hand, by Theorem~\ref{t.embed} and Theorem~\ref{t.0-1} in the game $G(\calp)$ neither of the players can attain the embeddability of a finitely generated group $H$ with unsolvable word problem.

\section{Further problems}\label{s.problems}

In this section we collected some problems representing directions in which the study of generic groups may be continued.

\textbf{Applications in algebra.} Similarly to the case of Abelian groups, one may consider various subspaces of $\calg$ defined by group properties. It is interesting in its own right to determine in a given subspace whether there is a comeager isomorphism class or not. More importantly, we think there may be a great potential in using the Baire category method in (a subspace of) $\calg$ to answer purely algebraic questions. For example, given countably many group properties $P_1,P_2,\dots$ (in the ordinary sense) to prove that there exists a group that satisfies every $P_i$ it suffices to show that the corresponding subsets $\calp_i$ are comeager in $\calg$.

It would be interesting to inspect the genericity of some extensively studied properties of finitely generated groups. For example, one may ask whether hyperbolicity is generic among groups generated by the first $n$ natural numbers.

\textbf{Strengthening existential closedness.} In \cite{ROBINSON} Barwise and Robinson introduced the notion of $K$-generic structures (where $K$ stands for a theory). In \cite{ROBINSON} Theorem 3.6 shows that $K$-genericity implies existential closedness. In the case of group theory, $K$-generic groups form a proper subclass of the class of existentially closed groups. See \cite[page 142]{ROBINSON} and \cite[Theorem 7]{MACINTYRE1}. \textbf{Is $K$-genericity generic in the sense of Baire category?}

\textbf{Continuous maps.} We defined the spaces $X_\Gamma$ for any $\forall_2$-theory $\Gamma$ in Section~\ref{s.the_space}. It would be interesting to study continuous maps between these spaces. For example, one may consider the space of all countably infinite group algebras on the universe $\nat$ and the continuous map $G\mapsto\rat[G]$. A well-known open problem of algebra is \emph{the semiprimitivity conjecture}, which says that every group algebra is semiprimitive. \textbf{Is it true that for the generic $G\in\calg$ the group algebra $\rat[G]$ is semiprimitive?}

\textbf{Comeager isomorphism classes.} As we have noted in Remark~\ref{r.general_0-1}, if the class of countably infinite models of a $\forall_2$-theory $\Gamma$ has the JEP, then isomorphism classes are either meager or comeager. A great ambition would be to characterize $\forall_2$-theories $\Gamma$ for which there is a comeager isomorphism class in $X_\Gamma$. A more realistic goal is to find sufficient or necessary conditions.

\end{document}